\documentclass[a4paper,12pt]{article}
\usepackage{amsmath,amsfonts,amssymb,amsthm}
\usepackage{authblk}
\usepackage{tikz}
\usepackage{geometry}

\usepackage[OT2,T1]{fontenc}
\newcommand\textcyr[1]{{\fontencoding{OT2}\fontfamily{wncyr}\selectfont #1}}

\renewcommand{\leq}{\leqslant}
\renewcommand{\bar}{\overline}
\renewcommand{\geq}{\geqslant}

\newcommand{\Np}{\mathbb{N}^{+}}
\newcommand{\N}{\mathbb{N}}
\newcommand{\defeq}{\mathrel{\mathop:}=}

\newcommand\z[1]{\draw (#1) node[minimum size=14pt, white, draw=black, fill=black] {\bf{0}};}
\newcommand\on[1]{\draw (#1) node[minimum size=14pt, draw=black] {\bf{1}};}

\newtheorem{thm}{Theorem}
\newtheorem{propn}[thm]{Proposition}
\newtheorem{lem}[thm]{Lemma}
\newtheorem{cor}[thm]{Corollary}
\newtheorem{conj}[thm]{Conjecture}
\theoremstyle{definition}
\newtheorem{defn}[thm]{Definition}
\newtheorem{xmpl}[thm]{Example}
\theoremstyle{remark}
\newtheorem{rmk}[thm]{Remark}

\title{On Long Arithmetic Progressions in Binary Morse-Like Words}

\author[1]{Ibai Aedo}
\author[1]{Uwe Grimm}
\author[1]{Yasushi Nagai}
\author[2]{Petra Staynova}
\affil[1]{School of Mathematics and Statistics, The Open University,\break 
Walton Hall, Milton Keynes MK7 6AA, UK}
\affil[2]{School of Computing and Engineering, University of Derby,
Kedleston Road, Derby DE22 1GB, UK}
\date{}

\begin{document}

\maketitle

\begin{abstract}
We present results on the existence of long arithmetic progressions in the Thue--Morse word and in a class of generalised Thue--Morse words. Our arguments are inspired by van der Waerden's proof for the existence of arbitrary long monochromatic arithmetic progressions in any finite colouring of the (positive) integers.
\end{abstract}

{\footnotesize\noindent Keywords: Combinatorics on words; binary languages; infinite words; bijective substitutions; Thue--Morse sequence; arithmetic progressions}

\section{Introduction}

The fixed point of the Prouhet--Thue--Morse (or Thue--Morse) substitution has numerous interesting properties. Many of these can be found in \cite{AS99,AS}, which is, respectively contains, an extensive survey of properties of the Thue--Morse language. 
The inspiration behind our results comes from one of the most well-known Ramsey-type theorems, namely van der Waerden's theorem. It states that any colouring of the integers contains arbitrarily long arithmetic progressions. 
Before we make this more precise, we introduce a couple of notions. 

\begin{defn}
Let $L$ and $d$ be positive integers. 
An \emph{arithmetic progression of difference\/ $d$ and length\/ $L$} is a finite sequence of integers 
\[
a_0, \ a_0+d,\ldots, \ a_0+(L{-}1)d.
\]
\end{defn}
\begin{defn}
Let $c$ be a positive integer. A \emph{colouring of a set of integers\/ $S\subseteq \N$ by\/ $c$ colours} is a map
$S\longrightarrow C$,
where $C$ is a finite set of $c$ distinct colours.  
\end{defn}

Now, we can state van der Waerden's theorem \cite{vdW,GR}.

\begin{thm}[van der Waerden \cite{vdW}]\label{vdW}
Let\/ $L$ and\/ $c$ be positive integers. 
There exists a positive integer\/ $N$ such that any\/ $c$-colouring of the segment\/ $\{1,2,\ldots,N\}$ contains a monochromatic arithmetic progression of length\/ $L$. 
\end{thm} 

By van der Waerden's theorem, if $v$ is an
infinite sequence over a finite alphabet $\mathcal{A}$
(that is, $v\in\mathcal{A}^{\mathbb{N}}$) then, for each
positive integer $L$, $v$ contains monochromatic arithmetic
progression of length $L$. Here, the existence
of an arithmetic progressions with
difference $d$ and length $L$ means that there is $n\in\mathbb{N}$ such that
$v_n=v_{n+d}=v_{n+2d}=\cdots=v_{n+(L-1)d}$.
We can then ask whether the word $v$ contains
arbitrary long arithmetic progressions for a
fixed positive difference $d$.
This question is closely related to the 
existence of infinite arithmetic progressions:
given $d$, is there $n\in\mathbb{N}$ such that
$v_n=v_{n+md}$ for each positive integer $m$?

If $v$ is a fixed point of a primitive constant-length substitution, 
the existence of infinite arithmetic progressions is related to the 
spectral theory of the corresponding dynamical system. 
Note that we can alternatively consider a bi-infinite sequence
$w\in\mathcal{A}^{\mathbb{Z}}$ that is a repetitive fixed
point of the same substitution and ask whether there exists
$n\in\mathbb{Z}$ such that $w_n=w_{n+md}$ for each $m\in\mathbb{Z}$.
This problem on bi-infinite sequences $w$ is equivalent to
the original problem on $v$, since both define the same language 
and hence the same shift space.

We recall the fact that
for a constant-length substitution, its dynamical system
(with $\mathbb{Z}$ shift-action) has pure point (discrete) spectrum 
if and only if the corresponding tiling dynamical system 
(with $\mathbb{R}$ translation-action) has pure point spectrum. Using this fact
and \cite[Theorem 5.1]{NAL}, we have the following.

\begin{thm}[{\cite[Theorem~5.1]{NAL}}]\label{Yasushi2}
Let $w$ be a repetitive fixed point of a primitive,
constant-length substitution.
Then\/ $w$ contains infinite arithmetic progressions if and only if\/ $w$ 
has pure point dynamical spectrum.
\end{thm}

As a corollary of van der Waerden's theorem, we have that  for any $n\in\N$, the word $0^n$ 
occurs within the Thue--Morse sequence as an arithmetic progression. 
This makes an interesting contrast with the well-known fact that Thue--Morse is strongly cube-free (or overlap-free).
Another interesting result \cite{AFF} states that \emph{any} binary word appears as an arithmetic progression within the Thue--Morse word.
The non-existence of infinite monochromatic arithmetic progressions within infinite words has been addressed in \cite{LPZ,BR} and
the problem of what types of finite words can appear as arithmetic progressions within substitution (and other) sequences has been considered in several contexts. 
The notion of arithmetic complexity, which generalises the subword complexity, has, for instance,
been studied in \cite{F03,F05,ACF,CF}. Other related questions on Thue--Morse type systems concern the sequence of values at multiples of an integer \cite{MSS}, Gowers uniformity norms \cite{Kon} or prefix palindromic lengths \cite{F19}.

This paper is organised as follows. We start by reconsidering arithmetic progressions in the Thue--Morse word \cite{AS99,AS} in Section~\ref{sec:tm}, re-establishing results of Parshina \cite{Olga1} in a different way. Rather than using binary arithmetic, our proofs are mainly based on exploiting the substitution structure of the Thue--Morse word. In addition to the long arithmetic progressions observed in \cite{Olga1} for differences $d=2^n-1$, we find another series of long progressions for differences $d=2^n+1$. Our results generalise to a set of binary bijective substitutions which are analysed in Section~\ref{sec:gentm}. We note that our generalised Thue--Morse substitutions \cite{K,Martin,BG10,BGG} are different from those considered by Parshina \cite{Olga2}. In Section~\ref{sec:bij}, we establish a bound for the existence of certain arithmetic progressions for general bijective substitutions on an arbitrary alphabet, before concluding in Section~\ref{sec:con} with some open questions and proposed directions for further research. 

\section{Arithmetic progressions in the Thue--Morse word}\label{sec:tm}

We use $\N$ to denote the set of non-negative integers and  $\Np=\N\setminus\{0\}$ for the set of positive integers. Apart from Section~\ref{sec:bij}, we will work with the binary alphabet $\mathcal{A}=\{0,1\}$ throughout this paper. For a finite word $w=w_{0}w_{1}\dots w_{n-1}\in\mathcal{A}^n$, we denote its length by $|w|=n$. For any finite or infinite word $w$, we denote by $w_i$ its letter at position $i<|w|$, and by $w_{[i,j)}$ its subword (factor) $w_{i}w_{i+1}\dots w_{j-1}$ of length $j-i$, for $0\leq i< j\leq |w|$, with analogous definitions for bi-infinite words.

We consider the Thue--Morse word $v\in\{0,1\}^{\N}$ arising from the substitution 
\begin{equation}\label{eq:TM}
  \theta\colon\quad \begin{array}{c}
  0\,\mapsto\, 01\,\\ 
  1\,\mapsto\, 10,
  \end{array}
\end{equation}
as the fixed point $v=\lim_{n\to\infty}\theta^n(0)$. Note that this substitution is bijective and symmetric under the `bar' operation that exchanges the two letters (so $\overline{a}=1-a$ for $a\in\{0,1\}$, referred to as a `bar-swap symmetry' in \cite{BG}), which also implies that $\bar{v}=\lim_{n\to\infty}\theta^n(1)$ is another fixed point word. The word $v=v_0v_1v_2\dots$ satisfies
\[
  v_{2i}=v_{i} \quad\text{and}\quad v_{2i+1}=\bar{v_{i}}
\]
for all $i\in\N$. The letter $v_i$ is thus $0$ if the binary expansion of $i$ contains an even number of $1$s, and $1$ otherwise. Also, $v$ is overlap-free, which means that, for any finite, non-empty word $w$, $v$ does not contain $www_0$ as a subword, where $w_0$ denotes the first letter of $w$. Note also that $\theta^{n}(a)$ is reflection-symmetric if $n$ is even, and antisymmetric (meaning that the reflected word is the image under the bar operation) if $n$ is odd.

For $n\in\N$, we have the well-known recursions
\begin{equation}\label{eq:rec}
     \theta^{n+1}(a)\, =\, 
     \theta^{n}(a)\,\overline{\theta^{n}(a)} 
     \, = \, 
     \theta^{n}(a)\,\theta^{n}(\overline{a}),
\end{equation}
as can easily be shown by induction. This implies the following property.

\begin{lem}\label{lem:TM}
For all\/ $m>n\in\Np$ and\/ $a\in\{0,1\}$, the word\/ $\theta^m(a)$ consists of a sequence of the two subwords\/ $w=\theta^n(a)$ and\/ $\bar{w}=\theta^n(\bar{a})$, arranged according to the sequence that corresponds to\/ $\theta^{m-n}(b)$ with the letters\/ $b$ and\/ $\bar{b}$ replaced by the words\/ $w$ and\/ $\bar{w}$.
\end{lem} 
\begin{proof}
Let $w=\theta^n(a)$ and $\bar{w}=\theta^n(\bar{a})$. Then,
$\theta(w)=w\bar{w}$ by the recursion \eqref{eq:rec}, which is the same form as the Thue--Morse substitution, now on the alphabet $\{w,\bar{w}\}$. Hence $\theta^m(a)=\theta^{m-n}(\theta^n(a))=\theta^{m-n}(w)$ is the sequence $\theta^{m-n}(b)$ with $b,\bar{b}$ replaced by $w,\bar{w}$. 
\end{proof}

\begin{propn}
The Thue--Morse word does not contain arbitrarily long monochromatic arithmetic 
progressions for any fixed difference\/ $d$. 
\end{propn}

\begin{proof}
This follows from the compactness of the associated dynamical system (see \cite{TAO} for general background) and from Theorem~\ref{Yasushi2}. 

Indeed, assume that the Thue--Morse word does contain arbitrarily long arithmetic progressions with difference $d$.  Note that if it contains a length $L$ arithmetic progression of colour 0, then it also contains a length $L$ arithmetic progression of colour 1. 
Thus, the colour of the arithmetic progression does not matter for this argument. 

Then, we can find positions $i_j$ for all $j\in\N$ such that $v_{i_j}$ is the first letter of a finite arithmetic progression of a fixed colour, difference $d$ and length $L_j$, satisfying $L_k>L_j$ for $k>j$, meaning that the lengths of the progressions are increasing.

Then $(T^{i_j}v)_{j\in\N}$, where $T$ denotes the shift map, is a sequence of words in the Thue--Morse dynamical system. By compactness, there exists a convergent subsequence, and the corresponding limit word contains an infinite monochromatic arithmetic progression of difference $d$ starting with its first letter.

Since the Thue--Morse sequence does not have pure point dynamical spectrum, this contradicts Theorem~\ref{Yasushi2}. Hence there are no arbitrarily long monochromatic arithmetic progressions of difference $d$ in the Thue--Morse word. 
\end{proof}

Thus, we may now introduce the following well-defined notion. 

\begin{defn}
For a positive integer $d$, let $A(d)$ be the maximum length of a monochromatic arithmetic progression of difference $d$ within the Thue--Morse word. 
\end{defn}

Since $\theta(v)=v$, we know that $A(2^n d)=A(d)$ for any $n\in\N$. In particular, since $v$ is overlap-free, this implies $A(2^n)=A(1)=2$ for all $n\in\N$. The following results shows that $A(d)=2$ holds only for differences $d$ that are powers of $2$.

\begin{lem}\label{lem:odd}
Let\/ $d>1$ be an odd integer. Then\/ $A(d)\geq 3$.
\end{lem}

\begin{proof}
Assume first that the binary expansion of $d$ contains an even number of $1$s. Since multiplication by $2$ conserves the number of $1$s, we have $v_0=v_d=v_{2d}=0$ and so $A(d)\geq 3$.

Now consider the case that the binary expansion of $d$ contains an odd number of $1$s, and hence at least $3$. Write $d=2^m+2^n+k$, with $m>n$ and $k<2^n$, so $k$ again contains an odd number of $1$s in its binary expansion. Let $i=2^{m+1}+2^{n}$, with $v_{i}=0$. Then $i+d=2^{m+1}+2^m+2^{n+1}+k$. If $m>n+1$, the number of $1$s in the binary expansion of $i+d$ is even. If $m=n+1$, then $i+d=2^{n+3}+k$ and so, the number of $1$s in its binary expansion is even too. Hence, $v_{i+d}=0$. Furthermore, $i+2d=2^{m+2}+2^{n+1}+2^{n}+2k$. If $2^n>2k$, the number of $1$s in the binary expansion of $i+2d$ is even. If $2^n\leq 2k<2^{n+1}$, we write $2k=2^n+t$, where the number of $1$s in the binary expansion of $t$ is even, and so,  $i+2d=2^{m+2}+2^{n+2}+t$ and its binary expansion has an even number of $1$s. Therefore, $v_{i+2d}=0$ and $A(d)\geq 3$.
\end{proof}

\begin{cor}
$A(d)=2$ if, and only if,\/ $d=2^n$ for some\/ $n\in\N$.
\end{cor}
\begin{proof}
Since $A(2^n d)=A(d)$ for all $n\in\mathbb{N}$ and $A(d)\geq 3$ for all odd $d>1$ by Lemma~\ref{lem:odd}, $A(d)=2$ implies that $d$ contains no odd prime factors. 
\end{proof}

Parshina proved the following result \cite{Olga1}, as well as a generalisation to similar sequences in larger alphabets \cite{Olga2,Olga3}. Her proofs for the Thue--Morse case \cite{Olga1} are based on a detailed analysis of binary arithmetic.

\begin{thm}[\cite{Olga1}]\label{thm:Olga}
For all $n\in\Np$, we have
\[
   \max_{d<2^n} A(d) \, = \, A(2^n-1) \,= \,
   \begin{cases}2^n+4, & \text{if\/ $2|n$,}\\
    2^n & \text{otherwise.}\end{cases}
\]
\end{thm}

In this paper, we will give exact expressions of $A(d)$ for certain values of $d$, first for the Thue--Morse sequence, and later also for generalised Thue--Morse sequences, which are  different generalisations from those considered in \cite{Olga1,Olga2,Olga3}. In particular, all our sequences are on a two-letter alphabet. Our results include a simple proof of the value of $A(2^n-1)$ stated in Theorem~\ref{thm:Olga}, but also identify a second series of long monochromatic arithmetic progressions in the Thue--Morse word, which becomes the `longest' in some cases, provided one considers the maximum over a different range for $d$.

Our proofs will follow van der Waerden's argument. For this, let us define the following block substitutions.

\begin{defn}\label{def:RTM-Block}
Let $\Theta$ be the block substitution on the alphabet $\{0,1\}$ defined by
\[
    \Theta\colon\quad 
     0\,\longmapsto\begin{matrix}0\,1\,\\1\,0,\end{matrix}\quad
     1\,\longmapsto\begin{matrix}1\,0\;\\0\,1.\end{matrix}
\]
\end{defn}

Iterating $\Theta$ on a single letter produces square blocks of size $2^n{\times} 2^n$, for instance
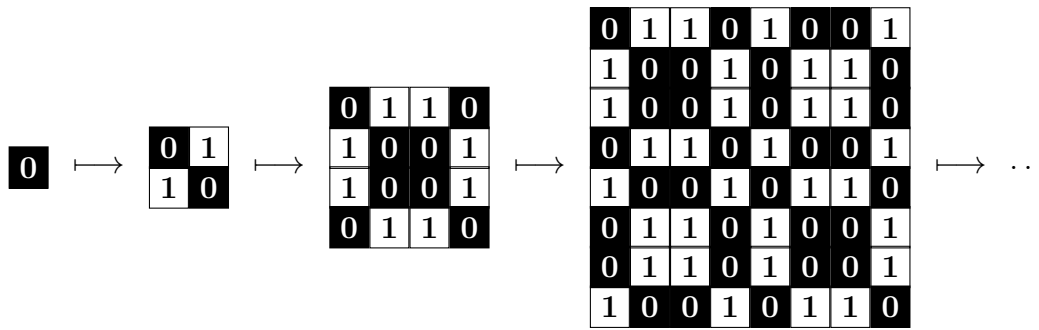
\[
\centerline{\begin{tikzpicture}[x=15pt,y=15pt]
\z{0,0} 
\node at (1.75,0) {$\longmapsto$};
\begin{scope}[shift={(3.5,-0.5)}]
\z{0,1}\on{1,1}
\on{0,0}\z{1,0}
\end{scope}
\node at (6.25,0) {$\longmapsto$};
\begin{scope}[shift={(8,-1.5)}]
\foreach \y in {0,3} {\z{0,\y}\on{1,\y}\on{2,\y}\z{3,\y}};
\foreach \y in {1,2} {\on{0,\y}\z{1,\y}\z{2,\y}\on{3,\y}};
\end{scope}
\node at (12.75,0) {$\longmapsto$};
\begin{scope}[shift={(14.5,-3.5)}]
\foreach \y in {1,2,4,7} {\z{0,\y}\on{1,\y}\on{2,\y}\z{3,\y}\on{4,\y}\z{5,\y}\z{6,\y}\on{7,\y}};
\foreach \y in {0,3,5,6}
{\on{0,\y}\z{1,\y}\z{2,\y}\on{3,\y}\z{4,\y}\on{5,\y}\on{6,\y}\z{7,\y}};
\end{scope}
\node at (23.25,0) {$\longmapsto$};
\node at (25,0) {$\dots$};
\end{tikzpicture}}
\]
where we used black (for $0$) and white (for $1$) squares to emphasise the block structure. Note that the blocks along both diagonals are always of the same colour.

\begin{lem}\label{lem:TM-Block}
For\/ $a\in\{0,1\}$ and\/ $n\in\Np$, the block\/ $\Theta^n(a)$, read row-wise from top to bottom, is the word\/ $\theta^{2n}(a)$ with\/ $\theta$ the Thue--Morse substitution of Eq.~\eqref{eq:TM}.
\end{lem}
\begin{proof}
This follows by induction from noticing that
\[
    \Theta\colon\quad 
     0\mapsto\begin{array}{c}
     \theta(0)\;\\[-2pt] \theta(1),\end{array}
     \quad
     1\longmapsto\begin{array}{c}
     \theta(1)\;\\[-2pt] \theta(0),\end{array}
\]
so that, read row-wise from the top, the image of $a$ under $\Theta$ is $\theta(a)\theta(\overline{a})=\theta^2(a)$.
\end{proof}

The images of letters under $\Theta^n$ have the 
following properties.

\begin{lem}\label{lemma:TM-prop}
For\/ $a\in\{ 0,1\}$ and\/ $n\in\Np$, the blocks\/ $\Theta^n(a)$ consists of only two types of row and column words, and are symmetric under reflection in either diagonals. All entries on the main diagonal are\/ $a$, while entries of the other diagonal are\/ $a$ for even\/ $n$ and\/ $\overline{a}$ otherwise.
\end{lem}
\begin{proof}
As shown in Lemma~\ref{lem:TM-Block}, $\Theta^n(a)$ when read row-wise from the top is the word $\theta^{2n}(a)=\theta^n(\theta^n(a))$. By Lemma~\ref{lem:TM}, this consists of $2^n$ words from $\{\theta^{n}(0),\theta^{n}(1)\}$. So each row is one of these two words.

The symmetry in the diagonals follows from the symmetry of the block inflation $\Theta$, which also implies that all columns are either $\theta^{n}(0)$ or $\theta^{n}(1)$.

It is obvious from the inflation rule that the elements of $\Theta^n(a)$ on the main diagonal are always $a$. On the other diagonal, note that $\Theta(a)$ has $\overline{a}$ while $\Theta^2(a)$ has $a$, which implies the claim.
\end{proof}

\begin{lem}\label{lem:ineq}
For all\/ $n\in\Np$, we have that\/ $A(2^n{\pm} 1)\geq 2^n$. For even $n$, we further have $A(2^n{-}1)\geq 2^n+2$.
\end{lem}

\begin{proof}
Consider the block $\Theta^n(a)$, which, when read row-wise from the top, is the word $\theta^{2n}(a)$. As shown in Lemma~\ref{lem:TM-Block}, all elements on the main diagonal are $a$, so we find that $A(2^n{+}1)\geq 2^n$. Similarly, the elements on the other diagonal are either all $a$ or $\bar{a}$, so 
we also have $A(2^n{-}1)\geq 2^n$. For even $n$, both diagonals have $a$ entries, and so the first and last letter of $\theta^{2n}(a)$ are also part of the arithmetic progression of difference $2^n{-}1$, which implies the claim.
\end{proof}

Before we establish the values for $A(2^n{\pm}1)$, we prove a useful result, which exploits the recognisability of the substitution; see \cite{TAO} and references therein for general background.

\begin{lem}\label{lem:rec}
For\/ $n>1$ and\/ $a\in\{0,1\}$, the word\/ $w=\theta^{n}(a)$ occurs in the Thue--Morse word either as the level-$n$ superword itself, or in the centre of two level-$n$ superwords\/ $\theta^{n}(\bar{a})\,\theta^{n}(\bar{a})$.

Furthermore, if $w$\/ is followed by the letter\/ $\bar{a}$, or if\/ $w$ is preceded by the letter\/ $a$ (for\/ $n$ odd) or\/ $\bar{a}$ (for\/ $n$ even), it is the level-$n$ superword.
\end{lem}

\begin{proof}
Clearly $w$ can occur as the level-$n$ superword. The second possibility arises from
\begin{align*}
\theta^{n}(\bar{a}\bar{a}) &=
\theta^n(\bar{a})\,\theta^{n}(\bar{a})\\
&=\theta^{n-1}(\bar{a}a)\,\theta^{n-1}(\bar{a}a)\\
&=\theta^{n-1}(\bar{a})\,\theta^{n-1}(a)\,\theta^{n-1}(\bar{a})\theta^{n-1}(a)\\
&=\theta^{n-1}(\bar{a})\, w\,\theta^{n-1}(a).
\end{align*}
To show that these are the only two possibilities, we use that
\[
  \theta^{n}(a) = \theta^{n-2}(a\bar{a}\bar{a}a)=
  \theta^{n-2}(a)\,\theta^{n-2}(\bar{a})\,\theta^{n-2}(\bar{a})\,\theta^{n-2}(a),
\]
which holds for all $n>1$. By recognisability, the two adjacent level-$(n{-}2)$ superwords $\theta^{n-2}(\bar{a})$ cannot belong to the same level-$(n\!-\!1)$ superword, so we know that $\theta^{n}(a)$ has to consist of two level-$(n{-}1)$ superwords, which only leaves the two possibilities, since all level-$(n{-}1)$ boundaries are determined. 

If $w=\theta^{n-1}(a)\,\theta^{n-1}(\bar{a})$ is followed by a letter $\bar{a}$, the next level-$(n{-}1)$ superword is determined to be $\theta^{n-1}(\bar{a})$, and the level-$n$ superword boundary has to fall between $w$ and the subsequent letter $a$, which shows that $w$ is the level-$n$ superword. The same happens when $w$ is preceded by the final letter of the superword $\theta^{n-1}(a)$, which is $a$ for odd $n$ and $\bar{a}$ for even $n$. 
\end{proof}

\begin{propn}\label{TM-long-plus1}
For all\/ $n>1$, we have that\/ $A(2^n{+}1)=2^n+2$. 
\end{propn}

\begin{proof}
We first show that there exist arithmetic progressions of length $2^n+2$. From the proof of Lemma~\ref{lem:ineq}, we already have an arithmetic progression of length $2^n$ in the word $w=\theta^{2n}(a)$, with the first and final letter being part of the progression. 

Now consider how many letters can be added at either end of the progression of length $2^n$ in the superword $w=\theta^{2n}(a)$. From Lemma~\ref{lem:rec}, we know that the word $w=\theta^{2n-1}(a)\,\theta^{2n-1}(\bar{a})$ and that these are the actual level-$(2n{-}1)$ superwords. There are for possibilities how this superword can be bordered by level-$(n{-}1)$ superwords: we can have $\theta^{2n-1}(b)\, w\, \theta^{2n-1}(c)$ with $b,c\in\{a,\bar{a}\}$. 

Since $d=2^n+1$, no element of the progression is in the level-$n$ superwords adjacent at either end. Since all superwords start or end with a level-$2$ superword $b\bar{b}\bar{b}b$, the next two members on either side would have to be the first and the second letter of the same superword $\theta^n(\bar{b})$, which however are different letters (where we use that $n>1$). This shows that the progression can at most be extended by one in either direction. Since all combinations of superwords on either side can appear, there are instances where the progression can be extended by exactly one step in both directions, showing that $A(2^n{+}1)\ge 2^n+2$.
 
It remains to be shown that this is the maximum length of a progression. Assume that we have a progression of length $L>2^n$. The elements in this progression hit each position in the superwords of level-$n$ at least once. Now, once we hit the first position of such a superword, the following members of the progression determine the sequence of level-$n$ superwords uniquely, which is the same sequence as that of the superword $w$. If there are at least $2^n$ terms in the progression following this position, they determine the level-$(2^n{-}1)$ superword by the second part of Lemma~\ref{lem:rec}, and hence we are back considering the word $w$ from above. If there are fewer terms left, we can use the previous member of the progression which hits in the last position of a level-$n$ superword, and determine the sequence of level-$n$ superwords preceding it in the progression. Again, this determines the level-$(2^n{-}1)$ superword by the second part of Lemma~\ref{lem:rec}, and we are back in the case considered above, showing that $L\le 2^n+2$. 
\end{proof}

Similarly, as mentioned above, we can rederive the value of $A(2^n{-}1)$ stated in Theorem~\ref{thm:Olga}.

\begin{propn}\label{TM-long-minus1}
For all\/ $n\in\N$, $n>1$, we have that
\[
   A(2^n{-}1) \, = \, \begin{cases}
   2^n+4, & \text{if\/ $2|n$},\\
   2^n, & \text{otherwise}.
   \end{cases}
 \]
\end{propn}
\begin{proof}
From Lemma~\ref{lem:ineq}, we already know that 
$A(2^n{-}1)\geq 2^n$ for $n$ odd and $A(2^n{-}1)\geq 2^n+2$ for $n$ even. 

Let us first consider the case that $n$ is odd. Since the superwords $\theta^n(a)$ are antisymmetric under reflection, their first and last letters differ. This means that once
our progression hits the first letter of a superword, it stops. Since addition by $2^n-1$ means that the elements in the progression cycle through all positions in the superword, we obtain the upper limit $A(2^n{-}1)\leq 2^n$, implying that $A(2^n{-}1)=2^n$.

Now consider the case of $n$ even. Here, the superwords are symmetric under reflection, so we can have two elements of the progression within one superword. Assume that this occurs for a superword $\theta^n(a)$ which has first and last letter $a$. If the progression continued to the left and to the right, the neighbouring superwords are determined by having the letter $a$ at the next two positions, which force both of them to be $\bar{a}$, and by symmetry this applies to either side of the superword, hence we obtain $\theta^n(\bar{a})\theta^n(\bar{a})\theta^n(a)\theta^n(\bar{a})\theta^n(\bar{a})$. Clearly, the word $\bar{a}\bar{a}a\bar{a}\bar{a}$ does not belong to the Thue--Morse language. This means that once the progression hits the  first and last letter, it can only be extended by at most one either way, so we obtain $A(2^n{-}1)\leq 2^n+4$. That this bound is attained can be seen by looking at $\theta^{2n}(a)$, which by Lemma~\ref{lem:ineq} contains a progression of length $2^n+2$ starting at ending with a superword $\theta^n(a)$ that contain two elements of the progression. The superwords either side of $\theta^{2n}(a)$ can be both $\theta^{2n}(\bar{a})$, since $\bar{a}a\bar{a}$ is in the Thue--Morse language. Hence the progression can be continued by one additional step to either direction, and the bound is attained.  
\end{proof}

The following two lemmas prove that there are no longer arithmetic progressions for differences up to powers of $2$.

\begin{lem}\label{lem:TM-bound-odd}
Let\/ $n\in\Np$ and\/ $0<k<2^{n-1}$ be both odd, and consider\/ $d=2^n-k$. Then\/ $A(d)\leq 2^n$. \end{lem}
\begin{proof}
Assume that there exists an arithmetic progression with difference $d$ of length $L>2^n$. Since 
$d$ is odd and hence coprime with $2^n$, looking at elements of the progression within superwords of length $2^n$ will meet every position in a superword, including the position $m=(k-1)/2$. However, the superwords of length $2^n$ for $n$ odd are 
antisymmetric under reflection, so 
\[
   \theta^n(a)_{m}= \overline{\theta^n_{2^n-1-m}(a)},
\]
which shows that not both $m$ and $2^n-1-m=m+d$ can be in the arithmetic progression, in contradiction to our assumption. This means that the progression cannot be longer than the number of rest classes modulo $2^n$, which establishes the claim.
\end{proof}

\begin{lem}\label{lem:TM-bound-gen}
Let\/ $n>1$ and let\/ $2<k<2^{n-1}$ be odd, and consider\/ $d=2^n-k$. Then\/ $A(d)\leq 2^n$. 
\end{lem}
\begin{proof}
Assume there exists an arithmetic progression with difference $d$ of length $L>2^n$. Since 
$d$ is odd and hence coprime with $2^n$, looking at elements of the progression within superwords of length $2^n$ will meet every position in a superword. There are precisely $k$ instances where two elements of the progression appear within the same the superwords of length $2^n$, and hence the corresponding letters within the superwords have to agree. Since $\theta^n(\bar{a})=\bar{\theta^n(a)}$, both superwords have to agree on all these positions. As a consequence, for $a\in\{0,1\}$ and the superword $\theta^n(a)$, the word consisting of its first $k$ letters
\[
  w \defeq \theta^n(a)_{[0,k)}
\]
also has to appear at the end of the superword, so $w=\theta^n(a)_{[2^n-k,2^n)}$. Since $k\geq 3$ and $\theta^n(a)$ starts with $a\bar{a}\bar{a}$, the word $w$ always contains a repeated letter and hence the level-$1$ superwords of length $2$ are uniquely determined. This results in a contradiction because the length of $w$ is odd, and the superword $\theta^n(a)$ thus cannot end in $w$, since the level-$1$ superword boundaries do not match. Hence $A(d)\leq 2^n$.
\end{proof}

Note that, in contrast to Lemma~\ref{lem:TM-bound-odd}, the result of Lemma~\ref{lem:TM-bound-gen} does not extend to the case $k=1$, since in this case there is only one instance of a word containing two elements of the arithmetic progression, and for even $n$ the superwords $\theta^n(a)$ start and end on $a$, so this can (and does) appear in a long arithmetic progression.

\section{Generalised Thue--Morse words}\label{sec:gentm}

Consider the generalised Thue--Morse substitution rules $\theta_{p,q}$ for $p,q\in\Np$ defined by \cite{BGG}
\begin{equation}\label{eq:tpq}
   \theta_{p,q}\colon\quad \begin{array}{l} 0\mapsto 0^p1^q\\ 
   1\mapsto 1^p0^q\end{array}
\end{equation}
where the original Thue--Morse substitution corresponds to $p=q=1$. These binary bijective substitutions share many properties with the Thue--Morse substitution. In particular, we still have the `bar-swap' symmetry $\theta^n(\bar{a})=\bar{\theta^n(a)}$. This implies that, once again, superwords are uniquely determined as soon as you know a single of its letters. Note that, however, the symmetry of superwords is only preserved when $p=q$, with superwords for even $n$ being symmetric while those for odd $n$ being antisymmetric under reflection. The other main change is that, rather than working modulo $2$, we now have to work modulo $Q\defeq p+q$. Also, it is clear from the substitution rule \eqref{eq:tpq} that the language of $\theta_{p,q}$ is $(Q{+}1)$-powerfree (in fact, $(Q+\varepsilon)$-powerfree for any $\varepsilon>0$), generalising the cube-freeness (overlap-freeness) of the Thue--Morse case.

We note that Parshina also considered generalised Thue--Morse words \cite{Olga2}, but in her work the generalisation is to larger alphabets. Here, we consider a generalisation of the Thue--Morse sequence along the lines of \cite{K,BG10,BGG}, restricting ourselves to the binary case.

Since the rule $\theta^2_{p,p}$ is symmetric under reflection, the corresponding language is reflection symmetric too. However, if $p\neq q$, reflection swaps the languages defined by
$\theta_{p,q}$ and $\theta_{q,p}$. As we shall now show, each of these languages itself is not reflection symmetric.

\begin{lem}\label{lem:nonsym}
For\/ $p\neq q$, the languages\/ $\mathcal{L}_{p,q}$ and $\mathcal{L}_{q,p}$ defined by the substitutions\/ $\theta_{p,q}$ and\/ $\theta_{q,p}$, respectively, are different, so\/ $\mathcal{L}_{p,q}\neq \mathcal{L}_{q,p}$. In particular, for\/ $a\in\{0,1\}$, the words\/ $\bar{a}\,a^p\,\bar{a}^{q+1}$ belong to $\mathcal{L}_{p,q}$ but not to\/ $\mathcal{L}_{q,p}$.
\end{lem}
\begin{proof}
It is easy to verify that $aa\bar{a}\in\mathcal{L}_{p,q}$ for any $p,q\in\Np$. Now,
\[
   \theta_{p,q}(aa\bar{a}) = a^p\,\bar{a}^q\,a^p\,\bar{a}^q\,\bar{a}^p\,a^q = a^p\,\bar{a}^q\,a^p\,\bar{a}^{p+q}\,a^q ,
\]
so $\bar{a}\,a^p\,\bar{a}^{q+1}\in\mathcal{L}_{p,q}$ for all $p\ne q$.
By reflection, $\bar{a}\,a^p\,\bar{a}^{q+1}\not\in\mathcal{L}_{q,p}$ is equivalent to  $\bar{a}^{q+1}\,a^p\,\bar{a}\not\in\mathcal{L}_{p,q}$, which we are going to show now.

Noting that $p\ne q$ and that $\mathcal{L}_{p,q}$ can only contain strings of the type $\bar{a}a^m\bar{a}$ for $m\in\{p,q,p{+}q\}$, it follows by recognisability that $a^p\bar{a}$ in $\bar{a}^{q+1}a^p\bar{a}$ has to be the start of the level-$1$ superword $\theta_{p,q}(a)$. However, it then has to be preceded by the level-$1$ superword that ends in $\bar{a}$, which is again $\theta_{p,q}(a)=a^p\bar{a}^q$. This is clearly impossible, establishing the claim.  
\end{proof}

\begin{rmk}\label{rem:nonsym}
Similarly, considering the word $a\bar{a}\bar{a}\in\mathcal{L}_{p,q}$, with
\[
\theta_{p,q}(a\bar{a}\bar{a}) = a^p\, \bar{a}^q\, \bar{a}^p\, a^q \,\bar{a}^p \,a^q
= a^p\, \bar{a}^{p+q}\, a^q\, \bar{a}^p\, a^q,
\]
we can show that $\bar{a}^{p+1}\,a^q\,\bar{a}\in\mathcal{L}_{p,q}$, but is not a word
in $\mathcal{L}_{q,p}$ for $p\ne q$.
\end{rmk}

Since all substitutions are binary bijective, it follows from \cite{BGG} that they are not pure point diffractive, which implies that they do not have pure point dynamical spectrum either. Hence, by Theorem~{Yasushi2}, they cannot contain infinitely long arithmetic progressions for any finite difference $d$. This means that we can again define the maximum length of an arithmetic progression.

\begin{defn}
For a positive integer $d$, let $A_{p,q}(d)$ denote the maximum length of a monochromatic arithmetic progression of difference $d$ within the generalised Thue--Morse word.
\end{defn}

As a direct consequence of the substitution structure and recognisability, we know that
$A_{p,q}(Q^nd)=A_{p,q}(d)$ holds for all $n\in\N$. In particular, this implies that $A_{p,q}(Q^n)=A_{p,q}(1)=Q$. We can again find long arithmetic progressions by considering a block substitution.

\begin{defn}\label{def:RTM-Block-gen}
Let $\Theta_{p,q}$ be the block substitution on the alphabet $\{0,1\}$ defined by 
\[
    \Theta_{p,q}\colon\quad 
     0\,\longmapsto\begin{array}{l}\left.\begin{array}{l@{}}0^p1^q\\[-2pt]
     0^p1^q\\[-4pt]
     \vdots\\
     0^p1^q\end{array}\right\} p \\[-2pt]
     \left.\begin{array}{l@{}}1^p0^q\\[-2pt]
     1^p0^q\\[-4pt]
     \vdots\\
     1^p0^q\end{array}\right\} q\end{array},
     \quad
      1\,\longmapsto\begin{array}{l}\left.\begin{array}{l@{}}1^p0^q\\[-2pt]
     1^p0^q\\[-4pt]
     \vdots\\
     1^p0^q\end{array}\right\} p \\[-2pt]
     \left.\begin{array}{l@{}}0^p1^q\\[-2pt]
     0^p1^q\\[-4pt]
     \vdots\\
     0^p1^q\end{array}\right\} q\end{array}.
\]
\end{defn}

The block substitution $\Theta_{p,q}$ maps a single letter $a$ to a $Q\times Q$ block of letters, which, when read line by line, coincides with the word $\theta^2_{p,q}(a)$. To illustrate the properties of $\Theta_{p,q}$, let us consider a couple of examples.

\begin{xmpl}\label{xmp:22-21}
We first consider and example where $p=q$, namely $\Theta_{2,2}$.
The first two substitution steps of the letter $0$ are as follows,

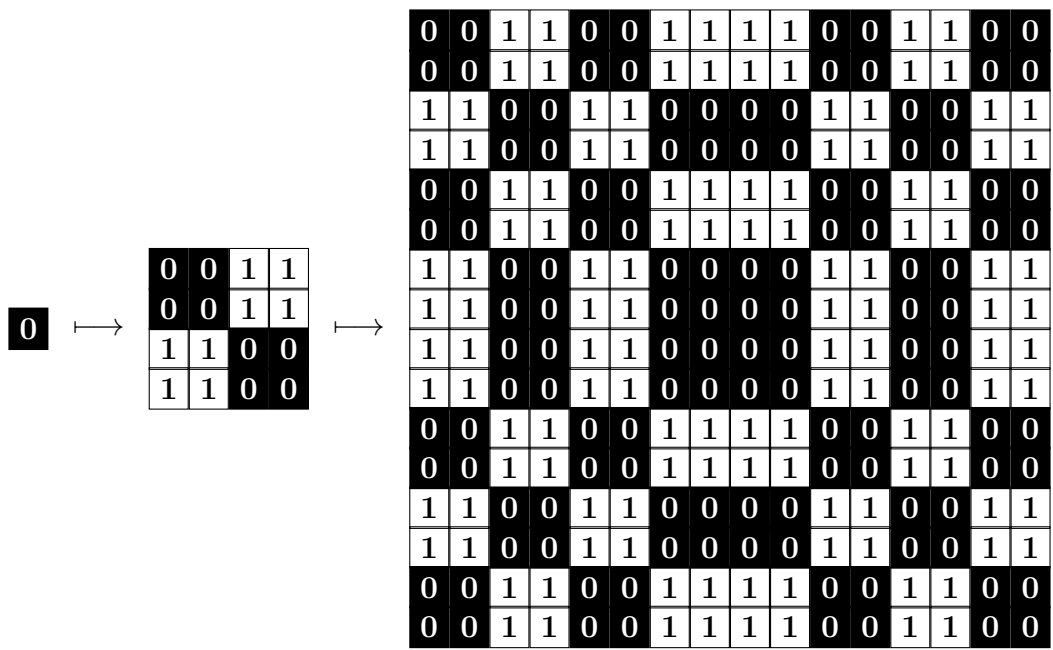
\[
\centerline{\begin{tikzpicture}[x=15pt,y=15pt]
\z{0,0} 
\node at (1.75,0) {$\longmapsto$};
\begin{scope}[shift={(3.5,-1.5)}]
\foreach \y in {2,3} {\z{0,\y}\z{1,\y}\on{2,\y}\on{3,\y}};
\foreach \y in {0,1} {\on{0,\y}\on{1,\y}\z{2,\y}\z{3,\y}};
\end{scope}
\node at (8.25,0) {$\longmapsto$};
\begin{scope}[shift={(10.0,-7.5)}]
\foreach \y in {0,1,4,5,10,11,14,15} {\z{0,\y}\z{1,\y}\on{2,\y}\on{3,\y}\z{4,\y}\z{5,\y}\on{6,\y}\on{7,\y}
\on{8,\y}\on{9,\y}\z{10,\y}\z{11,\y}\on{12,\y}\on{13,\y}\z{14,\y}\z{15,\y}};
\foreach \y in {2,3,6,7,8,9,12,13}
{\on{0,\y}\on{1,\y}\z{2,\y}\z{3,\y}\on{4,\y}\on{5,\y}\z{6,\y}\z{7,\y}
\z{8,\y}\z{9,\y}\on{10,\y}\on{11,\y}\z{12,\y}\z{13,\y}\on{14,\y}\on{15,\y}};
\end{scope}
\end{tikzpicture}}
\]
which is a similar structure as for the original Thue--Morse case. In particular, all squares along the diagonals are of the same colour.

The situation is different for $p\ne q$. Here, we consider the block substitution $\Theta_{2,1}$ as an example, which acts on a letter $0$ as
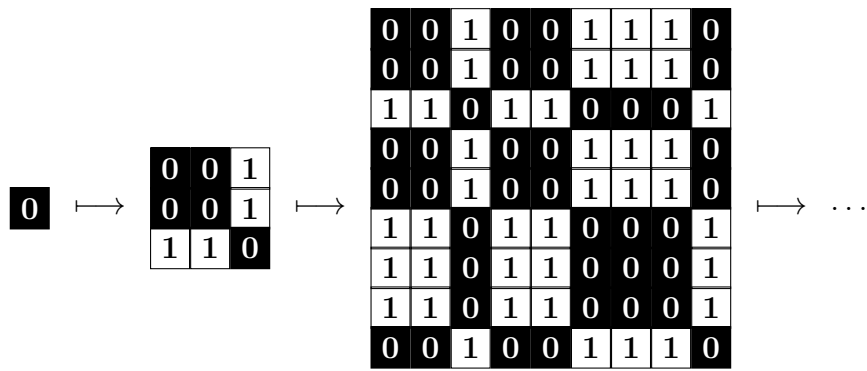
\[
\centerline{\begin{tikzpicture}[x=15pt,y=15pt]
\z{0,0} 
\node at (1.75,0) {$\longmapsto$};
\begin{scope}[shift={(3.5,-1.0)}]
\foreach \y in {1,2} {\z{0,\y}\z{1,\y}\on{2,\y}};
\foreach \y in {0} {\on{0,\y}\on{1,\y}\z{2,\y}};
\end{scope}
\node at (7.25,0) {$\longmapsto$};
\begin{scope}[shift={(9.0,-3.5)}]
\foreach \y in {0,4,5,7,8} {\z{0,\y}\z{1,\y}\on{2,\y}\z{3,\y}\z{4,\y}\on{5,\y}\on{6,\y}\on{7,\y}\z{8,\y}};
\foreach \y in {1,2,3,6}
{\on{0,\y}\on{1,\y}\z{2,\y}\on{3,\y}\on{4,\y}\z{5,\y}\z{6,\y}\z{7,\y}\on{8,\y}};
\end{scope}
\node at (18.75,0) {$\longmapsto$};
\node at (20.5,0) {$\dots$};
\end{tikzpicture}}
\]
Note that, while we retain the same letter along the main diagonal, this is no longer the case along the diagonal from the lower left to the top right. Considering the second inflation step shown above, it appears that there is a long monochromatic progression of difference $3^2-1=8$, starting from the central black square on the top row and moving down diagonally, and then continuing on from the final black square on the middle row. Indeed, we find that for $d=8$ the longest arithmetic progression has length $12$; however, this pattern does not persist for further inflation steps.
\end{xmpl}

As for the Thue--Morse case, the image of a letter $a$ has all entries $a$ along the main diagonal of this block. However, as illustrated in Example~\ref{xmp:22-21}, in general this is no longer the case for the other diagonal, except for the case that $p=q$, in which case the entries of this diagonal are all $\bar{a}$. This means that we obtain the existence of long arithmetic progressions, as in the Thue--Morse case, for $d=Q^n+1$ for all values of $p$ and $q$, while long arithmetic progressions for $d=Q^n-1$ may only exist if $p=q$. 

Noting that Lemma~\ref{lem:TM-Block} and Lemma~\ref{lemma:TM-prop} generalise in a straightforward manner, we obtain the following existence result for long arithmetic progressions in generalised Thue--Morse words, generalising the result of Lemma~\ref{lem:ineq}.

\begin{lem}\label{lem:ineq-gen}
For all\/ $n,p,q\in\Np$, $Q=p+q$, we have that\/ $A_{p,q}(Q^n{+}1)\geq Q^n$ and\/ $A_{p,p}(Q^n{-}1)\geq Q^n$. If\/ $n$ is even, we further have\/ $A_{p,p}(Q^n{-}1)\geq Q^n+2$.
\end{lem}
\begin{proof}
This follows by the same line of argument as for the Thue--Morse case in the proof of Lemma~\ref{lem:ineq}.  
\end{proof}

The following results implicitly use the fact that, as in the Thue--Morse case, a level-$n$ superword of $\theta_{p,q}$ within any word in $\mathcal{L}_{p,q}$ only occurs in certain ways. We obtain the following generalisation of Lemma~\ref{lem:rec}.

\begin{lem}
The word\/ $w=\theta_{p,q}^{n}(a)$ with\/ $a\in\{0,1\}$ and\/ $n>1$ occurs inside sufficiently long words in\/ $\mathcal{L}_{p,q}$ either as the level-$n$
superword itself, or, in the case when\/ $p=q$, in the centre of two level-$n$ superwords\/ $\theta_{p,p}^{n}(\bar{a})\theta_{p,p}^{n}(\bar{a})$. 

In the latter case\/ $p=q$, if $w$ is followed by the letter\/ $\bar{a}$ or preceded by the letter\/ $a$ (for\/ $n$ odd) or\/ $\bar{a}$ (for\/ $n$ even), it is the level-$n$ superword.
\end{lem}
\begin{proof} 
The proof is a straightforward generalisation from that of Lemma~\ref{lem:rec}. The only difference is that, for $p\ne q$, the word $w$ can only occur as the level-$n$ superword, because the level-$(n-1)$ superwords are determined and with $p\ne q$ they can only be combined to the level-$n$ superword in one way.
\end{proof}

\begin{propn}
For all\/ $n,p,q\in\Np$, $Q=p+q$, $n>1$, we have that
\[
  A_{p,q}(Q^n{+}1)=\begin{cases}
  Q^n+Q-2, & \text{if\/ $p>1$ and\/ $q>1$,}\\
  Q^n+Q-1, & \text{if\/ $q>p=1$ or\/ $p>q=1$,}\\
  Q^n+Q, & \text{if\/ $p=q=1$.}\end{cases} 
\]
\end{propn}

\begin{proof}
From the proof of Lemma~\ref{lem:ineq-gen}, we have an arithmetic progression of length $Q^n$ of the letter $a$ in the word 
\[
  w=\theta_{p,q}^{2n}(a)=\bigl(\theta_{p,q}^n(a)\bigr)^p\bigl(\theta_{p,q}^n(\bar{a})\bigr)^q\dots\,\bigl(\theta_{p,q}^n(\bar{a})\bigr)^p\bigl(\theta_{p,q}^n(a)\bigr)^q,
\]
with the first and final letter being part of the progression (as before, because we are looking at an even number of substitutions, the word $w$ starts and ends with the same letter). Note that, since the elements in the progression of difference $Q^n+1$ visit successive positions in superwords $\theta_{p,q}^n(a)$ in order, we know that, irrespective of where we start, once we hit the first letter of a superword $\theta_{p,q}^n(a)$ (which has to happen for any progression of length $Q^n$) the progression follows this same sequence, and the same backwards from when we hit the final position in a level-$n$ superword. Using the same argument as in the proof of Proposition~\ref{TM-long-plus1}, we conclude that any progression of length $L>Q^n$ has to include this superword.

Now consider how many letters can be added at either end of the progression of length $Q^n$ in the superword $w$. For $Q>2$, all four possibilities for this superword being bordered by level-$n$ superwords $u=\theta_{p,q}^{n}(a)$ or $\bar{u}=\theta_{p,q}^{n}(\bar{a})$ can occur, so we need to consider $w$ followed or preceded by either $u$ or $\bar{u}$.

If $w$ is followed by $u$, it is followed by $u^p$ and we can extend the arithmetic progression by exactly $p-1$ to the right. If it is followed by $\bar{u}$, we cannot extend at all unless $p=1$. For $p=1$, we can extend by exactly one step.

If $w$ is preceded by $u$ (for $n$ odd) or $\bar{u}$ (for $n$ even), we cannot extend at all unless $q=1$, in which case we can extend by precisely one step. If it is preceded by $\bar{u}$ (for $n$ odd) or $u$ (for $n$ even), we can extend by exactly $q-1$ steps to the left.

Choosing the combination with the longest available progression yields the result. 
\end{proof}

Note that for $p=q=1$ we recover the result of Proposition~\ref{TM-long-plus1}.

\begin{propn}
For all\/ $n,p\in\Np$, $Q=2p$, $n>1$, we have that
\[A_{p,p}(Q^n{-}1)=\begin{cases}
  Q^n, & \text{if\/ $n$ is odd}\\
  Q^n+Q, & \text{if\/ $n$ is even and $p>1$}\\
  Q^n+Q+2, & \text{if\/ $n$ is even and\/ $p=1$.}
\end{cases} 
\]
\end{propn}

\begin{proof}
From Lemma~\ref{lem:ineq-gen}, we already know that long arithmetic progressions for $d=Q^n-1$ exist, with $A(d)\geq Q^n$, within the superword 
\[
w=\theta_{p,p}^{2n}(a)=\bigl(\theta_{p,p}^n(a)\bigr)^p\bigl(\theta_{p,p}^n(\bar{a})\bigr)^p\dots\,\bigl(\theta_{p,p}^n(\bar{a})\bigr)^p\bigl(\theta_{p,p}^n(a)\bigr)^p.
\]
Accordingly, such a long progression visits every position in level-$n$ superwords.

For odd values of $n$, the superwords $\theta_{p,p}^n(b)$ start with $b$ and end on $\bar{b}$, so it is not possible to have the first and last letter in the same arithmetic progression. This implies that $A(d)\leq Q^n$, and hence $A(d)=Q^n$ in this case.

For even values of $n$, all superwords $\theta_{p,p}^n(b)$ start and end in the same letter, and hence we have $A(d)\geq Q^n+2$ as shown in Lemma~\ref{lem:ineq-gen}, with the first and last letter in the superword $w$ belonging to the arithmetic progression. What is left to consider is how far this can be extended on either side. The word $w$ can be preceded and succeeded by level-$n$ superwords $u=\theta_{p,p}^{n}(a)$ or $\bar{u}=\theta_{p,p}^n(\bar{a})$, where for $p=1$ one has to ensure cube-freeness.

If $w$ is succeeded by $u$ and hence by $u^p$, it can be extended by exactly $p-1$ steps. If is is succeeded by $\bar{u}$, no extension is possible, unless $p=1$ in which case you can extend by exactly one step. Due to symmetry of all these words for even $n$, the same argument applies at the other end, which completes the proof.
\end{proof}

\begin{propn}
For all\/ $n,p,q\in\Np$, with\/ $n>2$, $p\ne q$ and\/ $Q=p+q$, we have that $A_{p,q}(Q^n{-}1)\leq Q^n$.
\end{propn}
\begin{proof}
Assume to the contrary that a long arithmetic progression of difference $Q^n-1$ and length $L>Q^n$ exists. Then this progression contains a level-$n$ superword $w=\theta_{p,q}^n(a)$ with two instances of this progression, implying that the first and last letter of $w$ agree. If $n$ is odd, this is not possible, since $w$ starts with $a$ and ends on $\bar{a}$.

If $n>2$ is even, $w$ starts and ends with 
\[
\theta^2(a)=\theta(a)^p\theta(\bar{a})^q=(a^p\bar{a}^q)^p(\bar{a}^pa^q)^q.
\]
Since by bijectivity a single letter determines the superwords, we can read off the sequence of words to the left and to the right of the word $w$ with two instances of the progression, provided the progression extends.

Consider first the case $p>1$. Assume that the progression continues to the right of $w$. As we are considering the difference $d=Q^n-1$, we are effectively reading the word $w$ ``backwards'' to determine the sequence of superwords that is required. As mentioned above, $w$ ends on $\theta^2(a)$
which (since $p>1$) contains the word $\bar{a}\, a^{p}\,\bar{a}^{p+q}$. According to Lemma~\ref{lem:nonsym}, this word does not occur in $\mathcal{L}_{p,q}$, since we are considering the case that $p\ne q$. This implies that the sequence of superwords required to continue the progression for $Q^2$ steps to the right contains a subsequence that corresponds to the images of a word under $\theta_{p,q}$ that is not in the language $\mathcal{L}_{p,q}$, which is a contradiction. This means that the progression cannot continue to the right for more than $(q+1)Q$ steps at most.

An analogous arguments holds if you assume that the progression extends to the left, showing that it can at most continue for $pQ$ steps to the left. So the total length of the progression is at most $Q^2+Q<Q^n-1$ for $n>2$.

If $p=1$ and hence $q>1$, we can use the same arguments as above, based on the word $\bar{a}^{p+1}\, a^{q}\,\bar{a}$ from Remark~\ref{rem:nonsym}, which occurs within $\theta^2_{p,q}(a)$ in this case.
\end{proof}

So we have established the existence of long arithmetic progressions for all generalised Thue--Morse sequences for differences $d=Q^n+1$, as well as for differences $d=Q^n-1$ in the case that $p=q$. The obvious conjecture is that these are again the longest arithmetic progressions that you can find, up to the given difference, in these systems, which we state as a conjecture.

\begin{conj}\label{conj:gen-TM}
For all\/ $n,p,q\in\Np$, $Q=p+q$, $n>2$, we have that
\[
  \max_{d\leq Q^n+1}A_{p,q}(d)=\begin{cases}
  A_{p,q}(Q^n{-}1) = Q^n+Q+2, & \text{if\/ $p=q=1$ and\/ $n$ even,}\\
  A_{p,q}(Q^n{-}1) = Q^n+Q, & \text{if\/ $p=q>1$ and\/ $n$ even,}\\
  A_{p,q}(Q^n{+}1) = Q^n+Q, & \text{if\/ $p=q=1$ and\/ $n$ odd,}\\
  A_{p,q}(Q^n{+}1) = Q^n+Q-1, & \text{if\/ $q>p=1$ or\/ $p>q=1$,}\\
  A_{p,q}(Q^n{+}1) = Q^n+Q-2, & \text{if\/ $p,q>1$ and\/ $p\ne q$ or\/ $n$ odd.}
  \end{cases} 
\]
\end{conj}

To establish this conjecture, we would need to generalise the results of Lemmas~\ref{lem:TM-bound-odd} and \ref{lem:TM-bound-gen}. This is not straightforward, though, because we now have to consider differences $d=Q^n-k$ where we may have that $k$ is a non-trivial divisor of $Q$, in which case the argument that in a long arithmetic progression all rest classes modulo $Q^n$ appear is no longer applicable. 

The following lemma details the relations within superwords arising from an assumed existence of long arithmetic progressions.

\begin{lem}\label{lem:gen-TM-rel}
Consider the language of the generalised Thue--Morse substitution\/ $\theta_{p,q}$ with\/ $p+q=Q$.
For\/ $n\in\N$, $n>1$, and\/ $1<k<Q^{n-1}$, set\/ $s=\gcd(k,Q)$ and assume that\/ $s\neq Q$. 
If there exists a long arithmetic progression of difference\/ $d=Q^n-k$ and length\/ $L>Q^n/s$, the
level-$n$ superwords\/ $w=\theta^n_{p,q}(a)$ have to satisfy\/ $w_{r+\ell s}=w_{r+\ell s+d}$ for some\/
$0\leq r<s$ and for all\/ $0\leq \ell < k/s$. 
\end{lem}
\begin{proof}
We have $\gcd(d,Q)=\gcd(k,Q)=s$ and $r\equiv d\bmod Q$, so any such arithmetic progression of length $L$ visits all positions
\[
\bigl\{ r'\mid 0\leq r'<Q^n, r\equiv r'\bmod Q^n\bigr\} = 
\bigl\{r+\ell s\mid 0\leq \ell < \tfrac{Q^n}{s}\bigr\} 
\]
in a superword $w=\theta^n_{p,q}(a)$ for some letter $a\in\{0,1\}$ and for some $0\leq r<s$. Whenever there are two instances within a superword, the corresponding letters have to agree for either superword, since
$\theta^n_{p,q}(\bar{a})=\bar{\theta^{n}_{p,q}(a)}$. The condition for having two instances
within a superword is $r'+d<Q^n$, which means $r'<k$. With $r'=r+\ell s$, this results in 
$\ell<(k-r)/s$, and since $r<s$ this is equivalent to $\ell<k/s$, establishing the claim.
\end{proof}

\begin{propn}\label{prop:gen-TM-prime}
Consider the language of the generalised Thue--Morse substitution\/ $\theta_{p,q}$ for\/ $Q=p+q$ prime. Then, for\/ $n\in\N$, $n>1$, and\/ $Q<k<Q^{n-1}$, any arithmetic progression of difference\/ $d=Q^n-k$ has length\/ $L\leq Q^n$. 
\end{propn}
\begin{proof}
Since $Q$ is prime, we have that $\gcd(k,Q)=1$. From Lemma~\ref{lem:gen-TM-rel} we know that, if there exists an arithmetic progression of length $L>Q^n$, the superwords $w=\theta^n_{p,q}(a)$ have to satisfy $w_{[0,k)}=w_{[Q^n-k,Q^n)}$. If $k>Q$, this produces a contradiction, since the
final $Q$ letters of $w_{[0,k)}$ cannot be a valid level-$1$ superword, but $w$ has to end on a level-$1$ superword.
\end{proof}

Note that this lemma does not cover the differences $d=Q^n-k$ where $1<k<Q$, which we would need to establish the conjecture for prime values of $Q$ (except for $Q=2$ which brings us back to the Thue--Morse case). A partial result for $\min(p,q)=1$ is next, establishing Conjecture~\ref{conj:gen-TM} for this class.

\begin{propn}
Consider the language of the generalised Thue--Morse substitution\/ $\theta_{p,q}$ for\/ $\min(p,q)=1$ and\/ $Q=p+q$ prime. Then, for\/ $n\in\Np$ and any\/ $1<k<Q^{n-1}$, any arithmetic progression of difference\/ $d=Q^n-k$ has length\/ $L\leq Q^n$. 
\end{propn}
\begin{proof}
From Proposition~\ref{prop:gen-TM-prime}, we know that the claim holds from $k>Q$. 

If $p=1$, the level-$1$ superwords are of the form $a\bar{a}^q$ with $a\in\{0,1\}$. This means that, for $1<k<Q$, the superwords $w=\theta^n_{1,q}(a)$ start with $w_{[0,k)}=a\bar{a}^{k-1}$. Since $1\leq k-1<q$, this string of letters cannot occur at the end of the superword $w$.

Similarly, if $q=1$, the superwords $w=\theta^n_{p,1}(a)$ start with $w_{[0,k)}=a^k$. Since $k>q=1$, this string of letters cannot occur at the end of the superword $w$.
\end{proof}

\section{A result for general bijective substitutions}\label{sec:bij}

Clearly, our arguments for the existence of long arithmetic progressions observed for the Thue--Morse word and its generalisations considered in the previous section crucially depended on the particular arrangement of letters in the substitution rule. Even within the class of bijective substitutions (on an arbitrary alphabet), changing the order of letters will in general destroy the patterns we used above. However, if we restrict to binary bijective substitutions, we still have the following result.

\begin{propn}
  Let\/ $\theta$ be a binary bijective substitution of length\/ $Q$. Then, for any letter\/ $a$, the superwords\/ $\theta^{2n}(a)$ contain arithmetic progressions of distance\/ $Q^n+1$ and length $Q^n$. 
\end{propn}
\begin{proof}
  For a binary bijective substitution, the superword $\theta^{2n}(a)$ consists of a sequence of $Q^n$ superwords $\theta^n(a)$, $\theta^n(\bar{a})$, arranged in the order of letters in $\theta^n(a)$. As before, if you arrange these words in a square, all elements on the diagonal are the same letter $a$: whenever the $n$-th letter in  $\theta^n(a)$ is $a$, the superword in the $n$-th row is  $\theta^n(a)$ whose $n$-th letter is again $a$, and whenever it is $\bar{a}$, the $n$-th row is the superword $\theta^n(\bar{a})$, whose $n$-th letter is $\bar{\bar{a}}=a$. So we conclude $\theta^{2n}(a)_0=\theta^{2n}(a)_{Q^n+1}=\theta^{2n}(a)_{2Q^n+2}=\dots=\theta^{2n}_{Q^{2n}-1}=a$ which completes the proof.
\end{proof}  

In the remainder of this section, we state a partial result for general bijective (and hence constant-length) substitutions on a general alphabet $\mathcal{A}$.

Note that, due to linear repetitivity, any words that appear in the language defined by such a substitution occurs within any infinite word with a well-defined positive frequency. This means that we can alternatively formulate results in measure-theoretic terms.

\begin{thm}
Let\/ $\varrho$ be a primitive, bijective substitution rule with expansion factor\/ $Q$ on a general alphabet\/ $\mathcal{A}$. Assume that there is a positive integer\/ $k$ such that the word\/ $a^{k+1}$ is not in the language of\/ $\varrho$ for any letter\/ $a\in\mathcal{A}$. Let\/ $v\in\mathcal{A}^{\mathbb{Z}}$ be a repetitive fixed point of $\varrho$ and let\/ $X_{v}$ be the orbit closure of\/ $v$, and denote the unique shift-invariant probability measure on\/ $X_{v}$ by\/ $\mu$. Then the following statements hold.\smallskip

\noindent\emph{(1)} For some\/ $n>0$ and\/ $a\in\mathcal{A}$, the set
\begin{align*}
 C_{n,a}=\bigl\{u\in X_{v}\mid u_{0}=u_{n}=u_{2n}=\cdots =u_{kn}=a\bigr\}
\end{align*}
has positive measure: $\mu(C_{n,a})>0$.
If\/ $\mu(C_{n,a})>0$ and\/ $n=Q^{s}n'$ for some integers\/ $s\geq 0$ and\/ $n'>0$, then for some\/ $b\in\mathcal{A}$, we have\/ $\mu(C_{n',b})>0$.
Moreover, if\/ $\mu(C_{n,a})>0$ and the images of\/ $\varrho^{m}$ contain all letters in\/ $\mathcal{A}$, then\/ $\mu(C_{Q^{m}n,b})>0$ for each letter\/ $b\in\mathcal{A}$.\smallskip

\noindent\emph{(2)} 
Suppose that\/ $a\in\mathcal{A}$, $n>0$, and that\/ $C_{n,a}$ has positive measure. Assume that\/ $n$ and\/ $Q$ are coprime. Then there is\/ $m_{0}>0$ such that, for all\/ $m>m_{0}$, there are no arithmetic progressions of the letter\/ $a$ with difference\/ $d=Q^{m}\pm n$ and length\/ $L=Q^{m}+k.$

\noindent\emph{(3)}
Suppose that the substitution\/ $\rho$ is binary, that is,
the cardinality of\/ $\mathcal{A}$ is 2.
Then there are positive integers\/ $n$ and\/ $m_0$ such that,
if\/ $s\geq 0$ and\/ $m\geq m_0$, there are no arithmetic 
progressions for an arbitrary letter with difference\/ $Q^{m+s}\pm Q^sn$ and
length\/ $Q^m+k$.
\end{thm}

\begin{proof}
(1) By applying multiple Birkhoff recurrence theorem 
\cite[Theorem~2.6]{F} (see also \cite{dlL-W} for an application to tilings)
to $X_{v}$ and the powers of the shift map $T, T^{2}, T^{3},\ldots, T^{k}$, we see there are $w\in X_{v}$ and $n_{1}<n_{2}<\cdots$ such that, for any $i\in\{1,2,\ldots,k\}$,
\begin{align*}
  T^{in_{j}}w\longrightarrow w
\end{align*}
as $j\rightarrow\infty$. The set $\{u\in X_{v}\mid u_{0}=w_{0}\}$
is a neighbourhood of $w$ and we may assume that for each $j\in\Np$ and $i\in\{1,2,\ldots ,k\}$,
the sequence $T^{in_{j}}w$ is in this set, which means that
$w_{in_{j}}=w_{0}$. By setting $a=w_{0}$, we see the set $C_{n_{j},a}$ is not empty.
Since $\mu(C_{n_{j},a})$ is the frequency of a patch in $v$ and $v$ is linearly
repetitive, the measure $\mu(C_{n_{j},a})$ is positive.
         
If $n=Q^{s}n'$ and $\mu(C_{n,a})>0$, then there is $t\in\mathbb{Z}$ such that
$v_{t}=v_{t+n}=\cdots =v_{t+nk}=a$. Since $v_{t},v_{t+n},\ldots ,v_{t+kn}$ are relatively at the
same position in supertiles of length $Q^{s}$, bijectivity implies that these supertiles are the images of a common letter, which means that there is $t'\in\mathbb{Z}$ such that
$v_{t'}=v_{t'+n'}=\cdots =v_{t'+n'k}$. This implies that $\mu(C_{n',b})>0$ for
$b=v_{t'}$.

Suppose that 
$\mu(C_{n,a})>0$ and that for each $b\in\mathcal{A}$, the image $\varrho^{m}(b)$ contains every letter in $\mathcal{A}$. There is $t\in\mathbb{Z}$ such that $v_{t}=v_{t+n}=\cdots =v_{t+nk}$ and
the supertiles $\varrho^{m}(v_{t+sn})$, $s=1,2,\ldots ,k$, are all the same and separated by
the difference $Q^{m}n$. Since $\varrho^{m}(v_{t})$ contains every letter, for each $b\in\mathcal{A}$
there is an arithmetic progression of the letter $b$ with difference $Q^{m}n$ and length $k+1$, and thus $\mu(C_{Q^{m}n,k})>0$.\smallskip
         
\noindent (2) By unique ergodicity of $(X_{v},T)$, we have the convergence
\begin{align*}
 \lim_{m\rightarrow\infty}\frac{1}{m}\sum_{\ell=0}^{m-1}
    \chi^{}_{C_{n,a}}\circ T^{\ell}(v)=\mu(C_{n,a})>0,                
\end{align*}
where $\chi^{}_{C_{n,a}}$ denotes the characteristic function of $C_{n,a}$.
We can take $m_{0}>0$ such that, for all $m\geq m_{0}$, we have
\begin{align*}
 \sum_{\ell=0}^{Q^{m}-nk-1}\chi^{}_{C_{n,a}}\circ T^{\ell}(v)>0,
\end{align*}
which implies that there is $\ell_{0}\in\{0,1,\ldots ,Q^{m}-nk-1\}$ such that
$T^{\ell_{0}}(v)\in C_{n,a}$, which in turn implies that $v_{\ell_{0}}=v_{\ell_{0+d}}=\cdots =v_{\ell_{0}+kn}$.
        
To prove the statement, assume that there is $t\in\mathbb{Z}$ such that
$v_{t}=v_{t+d}=\cdots =v_{t+(L-1)d}=a$, and we will obtain a contradiction.
Since $d$ and $Q^{m}$ are coprime, there is $s_{0}\in\{0,1,\ldots ,Q^{m}-1\}$ such that
\begin{align*}
   t+s_{0}d \equiv
   \begin{cases}
    \ell_{0} \bmod Q^{m}&\text{if $d=Q^{m}+d$}\\
    \ell_{0}+kn \bmod Q^{m}&\text {if $d=Q^{m}-d$}.
   \end{cases}
\end{align*}
For each $s\in\{0,1,\ldots ,k\}$, we have
\begin{align*}
   t+(s_{0}+s)d \equiv
   \begin{cases}
    \ell_{0}+sn \bmod Q^{m}&\text{if $d=Q^{m}+d$}\\
    \ell_{0}+(k-s)n \bmod Q^{m}&\text {if $d=Q^{m}-d$}.
   \end{cases}
\end{align*}
If $d=Q^{m}+n$, 
the letter $v_{t+(s_{0}+s)d}$, coincides with $a=v_{\ell_{0}+sn}$ and these are at relatively the 
same position in the supertile of length $Q^{m}$, and so the supertiles that include
$v_{t+(s_{0}+s)d}$, $s\in\{0,1,\ldots, k\}$, are all the same type. However, this contradicts the
assumption that $b^{k+1}$ is not a legal word for any $b\in\mathcal{A}$.
Similarly, we obtain a contradiction for the case where $d=Q^{m}-n$.

\noindent (3) By part (1), there are $n>0$ and $a\in\mathcal{A}$
such that $n$ and $Q$ are coprime and 
$C_{n,a}$ has positive measure. Since $\rho$ is binary,
$C_{n,b}$ has positive measure for any letter $b$.
By part (2), there is a positive integer $m_0$ such that
for any $m\geq m_0$,
there are no arithmetic progressions of an arbitrary letter in $v$ with difference
$d'=Q^m\pm n$ and length $L=Q^m+k$.
Let $m\geq m_0$ and $s\geq 0$; we shall prove that there are no
arithmetic progressions with difference $d=Q^{m+s}\pm Q^sn$
and length $L=Q^m+k$. To this end, let us assume to the contrary
that there is $t\in\mathbb{Z}$ such that
\begin{align*}
    v_t=v_{t+d}=\cdots v_{t+(L-1)d}.
\end{align*}
Since $t=t+d=\cdots =t+(L-1)d$ mod $Q^s$, the letters
$v_t,v_{t+d},\ldots v_{t+(L-1)d}$ are at relatively the same
position in supertiles of length $Q^s$.
Due to bijectivity, this means that these supertiles are the images of 
the same letter. Hence there is $t'\in\mathbb{Z}$ such that
\begin{align*}
    v_{t'}=v_{t'+d'}=\cdots =v_{t'+(L-1)d'}.
\end{align*}
However, this contradicts the above observation that
there are no arithmetic progressions with difference $d'$
and length $L$, completing the proof.
\end{proof}

\section{Conclusions and outlook}\label{sec:con}

We investigated the occurrence of long arithmetic progressions in the Thue--Morse word and a class of generalised Thue--Morse words. Clearly, the existence of these long progressions of difference $d$ and length $L\approx d$ was linked to the structure of the underlying substitution, corresponding to a diagonal in the induced block substitution carrying the same letter. If we change the order of letters in the substitution while retaining bijectivity, this property persists in the binary case, but will be lost for larger alphabets and, in general, such long progressions should not be expected to exist. We currently do not have any stronger bounds on the progressions for these cases, but it is likely that infinite series of progressions where the length $L$ grows linearly with the difference $d$ may not exist in this case. It would be interesting to quantify the behaviour for such systems.

Other potential generalisations include the investigation of other substitution sequences, either with more letters (see \cite{Olga2} for results on a different class of generalised Thue--Morse sequences) or for non-bijective substitutions, or quantitative versions of normality results \cite{AFF,S1,S2,MS}.

\section*{Acknowledgements}
This research was supported by the Engineering and Physical Sciences Research Council (EPSRC) via Grant EP/S010335/1 (UG, YN) and by an Early Career Fellowship from the London Mathematical Society (PS).

\end{document}